\documentclass[11pt]{amsart}
\usepackage{amssymb,amstext}
\newtheorem{theorem}{Theorem}[section]
\newtheorem{lemma}[theorem]{Lemma}

\theoremstyle{definition}
\newtheorem{definition}[theorem]{Definition}

\newtheorem{cor}[theorem]{Corollary}

\theoremstyle{remark}

\numberwithin{equation}{section}

  \newcommand{\R}{\mathbb R}
  
  \newcommand{\C}{\mathbb C}
 \newcommand{\Z}{\mathbb{Z}}
 
  \newcommand{\Rr}{\mathcal{R}}
  \renewcommand{\L}{\mathcal{L}}
  \newcommand{\T}{\mathbb{T}}

\newcommand{\ts}{\hspace{0.5pt}}

\newcommand{\dd}{\,{\rm d}}

\newcommand{\MM}{\mathcal{M}(G)}

\newcommand{\MTB}{\mathcal{M}^{\infty}(G)}

\newcommand{\Oomega}{(X,G)}

\newcommand{\Ttheta}{(Y,G)}

\newcommand{\Hm}[1]{\leavevmode{\marginpar{\tiny%
$\hbox to 0mm{\hspace*{-0.5mm}$\leftarrow$\hss}%
\vcenter{\vrule depth 0.1mm height 0.1mm width \the\marginparwidth}%
\hbox to
0mm{\hss$\rightarrow$\hspace*{-0.5mm}}$\\\relax\raggedright #1}}}

\begin{document}
\title{Equicontinuous Delone dynamical  systems}

\author{Johannes Kellendonk}
\address{Universit\'e de Lyon, Universit\'e Claude Bernard Lyon 1,
Institut Camille Jordan, CNRS UMR 5208, 43 boulevard du 11 novembre
1918, F-69622 Villeurbanne cedex, France}
\email{kellendonk@math.univ-lyon1.fr}

\author{Daniel Lenz}
\address{Mathematisches Institut, Friedrich-Schiller Universit\"at Jena, Ernst-Abb\'{e} Platz~2, D-07743 Jena, Germany}
\email{ daniel.lenz@uni-jena.de }



\date{\today}

\begin{abstract}
We characterize equicontinuous Delone dynamical systems as those coming from Delone sets with strongly almost periodic Dirac combs.  Within the class of  systems with finite local complexity the only   equicontinuous systems  are then shown to be the  crystalline ones. On the other hand, within the class without finite local complexity, we exhibit examples of   equicontinuous  minimal  Delone dynamical  systems which are not crystalline.
 Our results solve the problem posed  by  Lagarias whether a   Delone set whose Dirac comb is strongly almost periodic must be crystalline.
\end{abstract}

\maketitle

\begin{center}
\emph{Dedicated to Robert V. Moody  on the occasion of his 70th birthday}
\end{center}

\bibliographystyle{amsalpha}

\tableofcontents

\newcommand{\bG}{\overline{G}}

\section*{Introduction}
The study of disordered systems is one of the most prominent issues in mathematics and physics today. A special  focus concerns   aperiodic order i.e.\  the very 'border' between order and disorder (see e.g. the monographs and conference proceedings \cite{BMed,Jan,Mbook,Pat,Sen}). This topic is particularly  relevant  for various reasons:  On the one hand, the actual discovery of physical substances \cite{INF,SBGC}, later called quasicrystals,  lying at this border has triggered enormous research activities in the experimental and theoretical description of low complexity systems.  On the other hand,  aperiodic order has  come up  in disguise in various contexts in
 purely conceptually motivated studies.

 As for such  conceptual studies of almost periodic order (in disguise), we mention work of Hedlund/Morse on Sturmian systems and  on complexity of aperiodic sequences  \cite{MH1,MH2}. Geometric analogues  have  recently  been  studied by Lagarias \cite{Lag0,Lag1} and Lagarias/Pleasants \cite{LP1,LP2}. Very loosely speaking, the corresponding results show that  in terms of suitable complexity notions there is a gap between the ordered world and the disordered world. Another approach to (dis)order in a spirit of Fourier analysis has been given in Meyer's work on harmonious sets \cite{Mey}. There, a  basic aim is to find and study a class of sets allowing for a Fourier type expansion. Meyer's considerations have been taken up  by  Moody \cite{Moo1,Moo2}. The corresponding results have become a cornerstone  in the study of  diffraction  aspects of  aperiodic order initiated by Hof \cite{Hof}. 

Of course, there are connections between the  geometry and complexity based approach  to  aperiodic order and the Fourier analytic side of the picture. For example, pure pointedness on the Fourier side implies zero entropy \cite{BLR}. A crucial link is provided by concepts of almost periodicity. On the one side almost periodicity is by its very definition a geometric concept about how the systems 'repeats itself'. On the other side, it is an important ingredient in all Fourier analytic considerations concerning pure point measures.
Almost periodicity has always played a role in the theory, see e.g. Solomyak \cite{Boris1} and Queffelec \cite{Que}. Still, it seems that only recently with the work of Baake/Moody \cite{BM} and subsequent work as e.g.  \cite{MS,Gou1,Gou2,LS,LR,Str}  the role of almost periodicity in diffraction theory  begins to be properly ascertained.

A most basic question in this context is whether it is possible to characterize order  by suitably strong forms of almost periodicity.  An affirmative answer to one version of this question is   given in the so-called C\'{o}rdoba Theorem \cite{Cord} (see \cite{Lag} for a generalization as well).
Another  version of this question  has been  posed  by Lagarias in \cite{Lag}. More precisely,  (with notions to be explained later), Problem 4.4 in \cite{Lag} asks  whether  a  Delone set whose  associated Dirac comb is strongly almost periodic is necessarily completely periodic.

The main  aim in this paper is  to answer this question. More specifically, we will show the following two points:

\begin{itemize}

\item For Delone sets with finite local complexity complete periodicity is indeed equivalent to strong almost periodicity of the associated Dirac comb (Corollary \ref{cor-Lagarias-true}).

\item There exist Delone sets (without finite local complexity) which are strongly almost periodic but not completely periodic (Corollary \ref{cor-Lagarias-wrong}).
\end{itemize}

So, in some  sense we show that the answer to the problem is both yes and no.

\smallskip

In order to avoid confusion let us mention that our notation differs from the notation of \cite{Lag} in the following way: In line with \cite{GdeL} and recent work on almost periodicity such as \cite{BM,MS,LS,LR} we use the term strongly  almost periodic measure for measures called uniformly  almost periodic in \cite{Lag}. We use the term crystalline or completely periodic for what is called ideal crystal in \cite{Lag}.

\smallskip

The result given in the  second point above  is treated in Section \ref{How-Failure} by providing a class of models based on the Kronecker flow on the two dimensional torus (Corollary \ref{cor-Lagarias-wrong}).

As for the result given in the  first point, it is quite clear that complete periodicity implies both finite local complexity and strong almost periodicity of the associated Dirac comb. Thus, the main work is to show the converse. Under the stronger assumption of Meyer property (instead of only finite local complexity) the result is already known. It  can be found in the recent work \cite{Stru2} and could also rather directly  be   inferred from the earlier  \cite{BLM} (see \cite{LR} as well for results in a very similar spirit).

Our approach  to the second point  relies on   three ingredients: The first ingredient  is the theory of equicontinuous systems discussed in Section \ref{Maximal}. We say that a Delone set is equicontinuous if
its associated dynamical system is equicontinuous.
 Equicontinuity of a dynamcial system is equivalent to the system being a rotation on a compact abelian  group (Ellis theorem, see Theorem~\ref{cor-equi}). A second ingredient are results of \cite{LR} discussed in Section \ref{Almost}  giving that a measure dynamical system is a rotation on a compact abelian group  if and only if the measures are strongly almost periodic (Theorem \ref{char-ap}). Combining these two pieces for Delone dynamical systems, we find that strong almost periodicity of the  associate Dirac comb is equivalent to equicontinuity of the system (Theorem \ref{thm-Delone-ec-general}).  Now, the third ingredient is  a  variant of a recent reasoning  of Barge/Olimb \cite{BO} presented in Section \ref{Proof}. It gives  that equicontinuity of a Delone dynamical system with finite local complexity is in fact equivalent to complete periodicity (Theorem \ref{thm-Delone-FLC}). Put together these considerations easily give  the desired result.


We present both, the  theory of equicontinuous systems as well as the considerations on strongly almost periodic Delone dynamical systems in somewhat more detail than needed for the actual answer to the question of Lagarias. The reason is that we   believe that these considerations may be of importance for future  study as well.

\section{Preliminaries on dynamical systems}
We deal with compact dynamical systems whose acting group is locally compact abelian.
An important special case are dynamical systems consisting of measures on the group.  The necessary notation is introduced in this section.

\bigskip

The space of continuous functions on a topological space
$X$ is denoted by $C(X)$ and  the subspace of continuous functions with compact
support by $C_c (X)$ and the space of continuous bounded functions by $C_b
(X)$.
We will deal with spaces $X$ which are  locally
compact $\sigma$-compact  Hausdorff spaces. Then,   $X$ carries the Borel $\sigma$-algebra generated
by all closed subsets of $X$ and by the Riesz-Markov representation
theorem (see e.g. \cite{Ped}), the set $\mathcal{M} (X)$ of all complex regular Borel
measures on $X$ can  be identified with the dual space $C_c
(X)^\ast$ of complex valued, linear functionals on $C_c(X)$ which are
continuous with respect to a suitable topology (see e.g. \cite[Ch.\ 6.5]{Ped} for details).  We then  write
$$\int_X f \dd\mu =
\mu(f)$$
 for $f\in C_c(X).$
The space $\mathcal{M} (X)$ carries the
\textit{vague topology}, which is  the weakest topology that makes all functionals
$\mu\mapsto \mu(\varphi)$, $\varphi\in C_c (X)$, continuous.

\smallskip

 The group operation of an locally compact abelian group
will  mostly be written additively as $+$.  The neutral element  will be denoted by $e$.
Now, let $G$ be a $\sigma$-compact locally compact
abelian group.
Whenever $G$
acts on the compact space $X$  by a continuous action
\begin{equation*}
 \alpha : \; G\times X \; \longrightarrow \; X
   \, , \quad (t,\omega) \, \mapsto \,  \alpha_t \omega\, ,
\end{equation*}
where $G\times X$ carries the product topology, the pair
$\Oomega$ is called a {\em topological dynamical system\/} over $G$. We will mostly suppress  the action $\alpha$ in our notation and write
$$ t\cdot \omega := \alpha_t \omega.$$
A dynamical system  $\Oomega$ is called {\em
minimal\/} if, for all $\omega\in X$, the $G$-orbit
$\{t\cdot \omega : t \in G\}$ is dense in $X$.

\begin{definition}(Factor)
  Let two topological dynamical systems\/ $\Oomega$ and\/ $\Ttheta$
  under the action of $G$ be given.  Then, $\Ttheta$ is called a\/
  {\em factor} of $\Oomega$, with factor map\/ $\varrho$, if\/
  $\varrho \! : X \longrightarrow Y $ is a continuous
  surjection with $\varrho (t\cdot \omega) = t \cdot
  (\varrho (\omega))$ for all\/ $\omega\in X$ and\/ $t\in G$.
\end{definition}

We will be concerned with special dynamical systems in which $X$ is a compact group. In order to simplify the notation, we introduce the following notation for these systems.

\begin{definition} (Rotation on a compact abelian group) A dynamical system $(X,G)$ is called a
rotation on a compact abelian group
if $X$ is a compact abelian group and the action of $G$ on $X$  is induced by a homomorphism $\iota : G\longrightarrow X$  such that $t \cdot x = \iota (t) +  x$ for all $t\in G$ and $ x\in X$.
\end{definition}

\newcommand{\UN}{\mathcal U}

\section{Equicontinuous systems and proximality}\label{Maximal}
In this section we recall some aspects of the theory of equicontinuous systems and proximality.
The importance of the proximality relation for the dynamical systems defined by tilings or Delone sets seems only to have recently been emphazised. In \cite{BargeDiamond} it is shown that topological closure of the proximality relation is a necessary and sufficient condition for (strong) Pisot substitution tilings to have pure point dynamical spectrum. This has been generalized in \cite{BargeKellendonk}. Related to proximality are the notions of asymptotic
tilings and augmentation which have led to finer topological invariants for substitution tilings \cite{BargeSmith}. Whereas the above mentionned work concerns the case in which the proximality relation is non-trivial, we consider here the
opposite case, namely when this relation is trivial.

We refer the reader to Auslander's book \cite{Aus} for more detailled information.
The theory is cast within the framework of uniform structures
instead of more conventional (but more restrictive) metric topologies.

\smallskip

A uniform structure $\UN$, see \cite{Kel} for details, is a family of subsets of $X\times X$ satisfying
\begin{enumerate}
\item the diagonal $\Delta = \{ (x, x) : x \in X \}$ is contained in all $U\subset \UN$,
\item with $U$ any superset $V\supset U$ is contained in $\UN$,
\item with $U$ and $V$ also $U\cap V$ belong to $\UN$,
\item with $U$ also $\{(y,x):(x,y)\in U\}$ belongs to $\UN$,
\item for all $U \in\UN$ exists $V \in \UN$  such that, whenever $(x, y),(y, z)\in V$, then $(x, z) \in U$.
\end{enumerate}
A uniform structure $\UN$ generates as unique toplogy such that for each $x\in X$ the sets $V_x :=\{ y : (x,y)\in V\}$ (for $V\in \UN$) are a neighborhood basis for $x$.  A base of a uniformity structure $\UN$ is any subfamily $\UN'\subset \UN$ such that all members of $\UN$ are supersets of some menber of $\UN'$.

\smallskip

Uniform structures have already been used in tiling theory. In fact,  Schlottmann  develops basic theory of  points sets (with finite local complexity)  on locally compact abelian groups in the framework of uniform structures in \cite{Schl}. For arbitrary point sets a discussion can be found in \cite{BL} (see \cite{BLM} as well). Note also that the construction of a cut and project scheme  out of a suitable autocorrelation measure  in \cite{BM} is  based on uniform structures.

The reader who is not familiar with uniform structures may at first consider the case of a metric space
whose uniformity structure has a base
given by the $\epsilon$-diagonals $\Delta_\epsilon := \{ (x,y)\in X\times X: d(x,y)<\epsilon\}$.
A map $f:X\to X$ is uniformly continuous (with respect to a given uniformity structure $\UN$ of $X$)
if for all $V\in\UN$ we find $U\in\UN$ such that $(f\times f) (U)\subset V$.
A family of maps $\mathcal F$ from $X$ to itself is called equicontinuous if for all $V\in\UN$ we find $U\in\UN$ such that for all $f\in \mathcal F$: $(f\times f)(U)\subset V$. These notions reduce to the conventional ones in a metric space provided one used the fundamental system given by the $\epsilon$-diagonals. On a compact space all uniformity structures coincide.

\bigskip

Consider a minimal dynamical system $(X,G)$ where  $X$ is a compact Hausdorff space and $G$ a locally compact abelian group acting by $\alpha$ on $X$.
If the action is free then $X$ can be seen as a compactification of $G$: it is the completion of one orbit and this orbit is a copy of $G$. One might ask when is $X$ a group compactification, that is, when does $X$ carry a group structure such that the orbit is a subgroup isomorphic to $G$, or, in other words, when is $(X,G)$ a rotation on a compact abelian group?
 The question of when this happens is to the equicontinuity of the action.
\begin{definition}
The dynamical system $(X,G,\alpha)$ is called equicontinuous
if the family of homeomorphisms $\{\alpha_t \}_{t\in G}$ is equicontinuous.
\end{definition}
This reduces to the usual definition if $(X,d)$ is a metric space: $$\forall \epsilon\exists \delta \forall x,y\in X: d(x,y)<\delta \Rightarrow \forall t\in G: d(\alpha_t(x),\alpha_t(y))<\epsilon.$$
An equicontinuous metrizable system admits always an invariant metric which induces the same topology. Indeed one can just take
$$\overline{d}(x,y) := \sup_{t\in G} d(t \cdot x,t \cdot y).$$
Likewise, any compact metrizable abelian group $\T$ admits a left invariant metric: Whenever $d$ is a metric  then
$\overline{d}(x,y) := \sup_{t\in \T} d(x+t,y+t)$ is a metric on $\T$ which induces the same topology and
 is $G$-invariant.

We next provide a  useful necessary condition for a system to be equicontinuous.
\begin{definition}
[Proximality] Consider a $G$ action on a compact Hausdorff space $X$.
Two points $x,y\in X$ are proximal if for all $U\in \UN $ there exists a $t\in G$ such that $(t\cdot x,t\cdot y)$ belongs to  $U$.
\end{definition}
In the context of compact metric spaces $(X,d)$ this translates into saying that $x,y$ are proximal whenever
$$\inf_{t\in G} d(t\cdot x,t\cdot y)=0.$$
\begin{cor} If $(X,G)$ is equicontinuous then
the proximal relation is trivial.
\end{cor}
Moreover, a system with non-trivial proximality relation cannot carry an invariant metric which is compatible with the topology.

An equicontinuous dynamical system need not be minimal but a transitive equicontinuous dynamical system is always minimal. So in the context of Delone and tiling dynamical  systems equicontinuous systems are always minimal.
The following theorem due to Ellis is of great importance for our results. Recall that we suppose that
$(X,G)$ is minimal.
\begin{theorem}
\label{cor-equi}
$(X,G)$ 
is conjugate to a minimal rotation on a compact abelian group if and only if it is equicontinuous.
\end{theorem}
If $(X,G)$ is equicontinuous the group structure on $X$ arises as follows: given any point $x_0\in X$ the operation $t_1\cdot x_0 + t_2\cdot x_0 := (t_1 + t_2)\cdot x_0$ extends to an addition in $X$ so that
$X$ becomes a group with $x_0$ as neutral element. Conversely any rotation on a compact abelian group  is obviously  equicontinuous.

To present the following characterization of equicontinuity in the context of metrizable dynamical systems we need some further  notation:  Let $(X,G)$ be a dynamical system and $d$ a metric on $X$. Then, the $\epsilon$-ball around $x\in X$ is denoted by $B_\epsilon (x)$. The elements of
 $$\Rr(x,\epsilon):=\{t\in G: t\cdot x\in B_\epsilon(x)\}$$
are called \textit{return vectors} to $B_\epsilon (x)$.
 Now, $(X,G)$ is called
\textit{uniformly almost periodic} if, for any $\epsilon>0$ the joint set of return vectors to $\epsilon$-balls, given by
$$A = \bigcap_{x\in X} \Rr(x,\epsilon)$$
is relatively dense (i.e.\ there exists a compact $K$ with $A + K = G$).
\begin{theorem}[\cite{Aus}]\label{thm-Auslander-equi}
$(X,G)$ is equicontinuous if an only if it is uniformly almost periodic.
\end{theorem}

\section{Almost periodic measures on locally compact abelian groups}\label{Almost}
We will be concerned with dynamical systems built from measures. We will show that equicontinuity of such a system is equivalent to  almost periodicity of the underlying measure (Theorem \ref{char-ap}). This provides a characterization of equicontinuity in this framework.  The considerations of this section can be understood as a (slight) reformulation of results obtained in \cite{LR}. Alternatively,
they could  - at least partly - be based on Theorem \ref{thm-Auslander-equi}.  A thorough study of measure dynamical systems in the framework of aperiodic order and diffraction can be found in \cite{BL,BL2,Len} to which we refer for further details.

\bigskip

Let $G$ be a locally compact $\sigma$-compact  abelian group.
A measure $\nu\in \MM $ is {\em translation bounded\/} if for some (and then all)  open non empty relatively compact  set $V$ in $G$ there exists a $C = C_V\geq 0$   with
\begin{equation*}\label{mcv} |\nu| (t+V) \leq C
\end{equation*}
for every $t\in G$. Here,  $|\nu|$ is the total variation measure of $\nu$.
 The set of all
translation bounded measures is denoted by $\MTB$.  As a subset of
$\MM$, it carries the vague topology.   There is an obvious action of $G$ on $\MTB$,  given by
\begin{equation*}
\; G\times  \MTB \; \longrightarrow \; \MTB
   \, , \quad (t,\nu) \, \mapsto \, \alpha^{}_t\ts \nu
   \quad \mbox{with} \quad (\alpha^{}_t \ts \nu)(\varphi) \, := \,
   \nu(\delta_{-t} \ast \varphi)
\end{equation*}
for $\varphi \in C_c (G)$.  Here, $\delta_t$ denotes the unit point
mass at $t \in G$ and the convolution $\omega \ast \varphi$ between
$\varphi\in C_c (G)$ and $\omega \in \MTB$ is defined by
\[ \omega \ast \varphi (s) := \int \varphi (s - u) \dd \omega(u).\]

It is not hard to see that this action is continuous when restricted to a
compact subset of $\MTB$ (see e.g. \cite{BL}).

\begin{definition}
  $\Oomega$ is called a dynamical system on the translation bounded
  measures on\/ $G$  {\rm (TMDS)} if
  $X$ is a compact  $\alpha$-invariant subset of\/ $\MTB$.
 \end{definition}

Every translation bounded measure $\nu$ gives rise to a (TMDS)
$(X(\nu),G)$, where
\[ X(\nu):=\overline{\{\alpha_t\nu : t\in G\}}.\]
In fact, if $\nu$ satisfies  the inequality $|\nu| (t + V)\leq C$ for some nonempty open relatively compact $V$ in $G$ and all $t\in G$, this inequality  will be true  for all $\mu \in X(\nu)$ and this shows the desired compactness by Theorem 2 of \cite{BL}.

\smallskip

As usual $\varphi \in C_b (G)$ is called {\em almost periodic} (in the
sense of Bohr) if, for every $\epsilon >0$, the set of $t\in G$ with
$\|\delta_t \ast \varphi - \varphi \|_\infty \leq \epsilon$ is
relatively dense in $G$. By standard reasoning this is equivalent to
$\{\delta_t \ast \varphi : t\in G\}$ being relatively compact in in
$C_b (G)$ (see e.g. \cite{Zai}).

\begin{definition}
A translation bounded measure $\nu$ is called strongly almost periodic
  if $\nu \ast \varphi$ is almost periodic (in the Bohr sense) for
  every $\varphi \in C_c (G)$.
\end{definition}

To deal with almost periodic functions  it is useful to introduce the strong topology. To do so, we recall that for $\nu\in \MTB$ and
$\varphi \in C_c (G)$, the convolution $\nu \ast \varphi$ belongs to
$C_b (G)$.  The space $C_b (G)$ is equipped with the supremum norm.
The strong topology is  the weakest topology on
$\MTB$ such that all maps
$$
\MTB\longrightarrow C_b (G), \;\:\nu\mapsto \nu\ast \varphi,  $$
are continuous.

\begin{theorem}\label{char-ap} Let $G$ be a locally compact $\sigma$-compact abelian group. 
Let $\nu \in \MTB$ be given.  The following assertions are equivalent:
\begin{itemize}
\item[(i)]  The measure $\nu$ is strongly almost periodic (i.e.\ $\nu \ast
  \varphi$ is Bohr almost periodic for every $\varphi \in C_c (G)$).

\item[(ii)] $\{\alpha_t\nu : t\in G\}$ is relatively compact in the strong topology.

\item[(iii)] The topological space $X(\nu)$   is a topological group
  with addition $\dotplus$ satisfying $\alpha_s \nu \dotplus \alpha_t \nu =
  \alpha_{s+t} \nu$ for all $ s,t\in G$.
\item[(iv)] The dynamical system $(X(\nu),G)$ is a rotation on a compact abelian group.
\item[(v)] The hull $(X (\nu), G)$ is an equicontinuous dynamical system.
\end{itemize}

\end{theorem}
\begin{proof} The equivalence of (i), (ii) and (iii) is shown in Lemma
  4.2 of \cite{LR}. The equivalence of (iii) and (iv) is immediate
  from the definition of rotation on a compact abelian group. The
  equivalence of (iv)  and (v) is Theorem \ref{cor-equi}.
\end{proof}

Factors of equicontinuous dynamical systems are equicontinuous as well. In our context the following variant will be of  use in Section \ref{How-Failure}.

\begin{cor}\label{cor-ap}
 Let $\nu$ be a translation bounded measure on
  $G$. Then, the following assertions are equivalent:
\begin{itemize}
\item[(i)] The measure $\nu$ is strongly almost periodic.

\item[(ii)] There exists a rotation of a compact abelian group  $(\T,G)$ together with a continuous map $\pi : \T\longrightarrow \MTB$ with  $\pi (e ) = \nu$ and  $\pi (t \cdot \xi ) = t \cdot\pi (\xi)$ for all $\xi\in \T$ and $t\in G$.
\end{itemize}
In this case $\pi (\T) = X (\nu)$ and $\pi$ is a factor map as well as a group homomorphism.
\end{cor}
\begin{proof}  In the situation of (ii) we must have that $\pi$ is a factor map with  $\pi (\T) = X(\nu)$ as $\T$ is compact and $\pi$ is continuous. Now
the equivalence of (i) and (ii) follows from Theorem 5.1 of \cite{LR}. The fact, that $\pi$ is a group homomorphism now  follows from Lemma 5.2 of \cite{LR}.
\end{proof}

\textbf{Remark.} The almost periodic measure dynamical systems
considered above are of interest in the study of diffraction. In fact,
they exhibit  pure point diffraction and  pure point dynamical
spectrum with continuous eigenfunctions. Details can be found in
\cite{LR} (see \cite{BLM,LS,Len} for related arguments on diffraction
as well). For a further study of a particular subclass of systems and a characterization in the context of cut and project schemes we refer the reader to \cite{Stru3}.

\section{Application to Delone dynamical systems}\label{Application}
The previous sections dealt with first  general dynamical systems and then  measure dynamical systems. In this section we specialize even further and consider  Delone dynamical systems.   These systems are made of suitable uniformly discrete subsets of the underlying group. Such subsets can be considered as translation bounded measures in a canonical way.  We will introduce the necessary notation to deal with these systems and to state Lagarias question. We will then go on and
 characterize what it means for such a system to be equicontinuous in Theorem \ref{thm-Delone-ec}. This provides a main step in our dealing with the Lagarias question. Our strategy to view Delone sets as measures has the advantage that it allows us to treat colored Delone sets with no extra effort. This is discussed at the end of the section.

\bigskip

Let $G$ be a locally compact abelian group.  We will deal with subsets $\L$ of $G$.
A subset $\L$ of $G$ is  called \textit{uniformly discrete} if there exists an open neighborhood $U$ of the identity  in $G$ such that
$$(x + U) \cap (y + U) = \emptyset$$
for all $x,y\in L$ with $x\neq y$. A subset $\L$ of $G$ is called \textit{relatively dense} if there exists a compact neighborhood $K$ of the identity of $G$ such that
$$ G = \bigcup_{x\in \L} (x + K).$$
A subset $\L$ of $G$ is called a \textit{Delone set} if it is both uniformly discrete and relatively dense.
Any uniformly discrete set $\L$ (and hence any Delone set) in $G$ can naturally be identified with the translation bounded measure
$$\delta_L := \sum_{x\in \L} \delta_x,$$
where $\delta_x$ denotes the unit point may at $x\in G$.  The \textit{hull} of $X(\delta_\L)$ in the vague topology then consists of measures of the form  $\delta_{M}$ with $M\subset G$ uniformly discrete. In the sequel we will identify such measures with the underlying sets and then also  write $X (\L)$ instead of $X(\delta_\L)$. We call $(X(\L), G)$   a \textit{Delone dynamical} system if $\L$ is a Delone set. As discussed e.g. in \cite{BL} the topology on $X$ inherited in this way  from the vague topology  on the measures  can also be obtained from a uniform structure on the set  $\mathcal{P}$ of all uniformly discrete point sets  on $G$.
Namely, for $K\subset G$ compact and $V\!$ a
neighbourhood of $e$ in $G$, we set
\begin{equation*}
   U_{K,V} \; := \; \{(\L_1,\L_2) \in \mathcal{P} \times \mathcal{P} :
   \L_1 \cap K \subset \L_2 + V \mbox{ and } \L_2 \cap K \subset \L_1 + V\}.
\end{equation*}
Then, the set $\UN$ consisting of all $U_{K,V}$ with $K\subset G$ compact and $V$ a neighborhood of $e\in G$, will provide a uniform structure. The induced topology is called the \textit{local rubber topology} in \cite{BL} and shown to agree   with the topology induced  by the vague topology on measures.

Besides the hull of $\L$ we will also need the \textit{canonical transversal}  given by
$$\Xi(\L) = \{\L'\in X: e\in\L\}.$$
This set can easily be seen to be the  closure of the set $\{\L-x:x\in\L\}$. For this reason it is also sometimes denoted as  the  \textit{discrete hull} of $\L$. Of course, the discrete hull of $\L$ is a subset of the hull $X$ and as such inherits a topology.

Whenever $\L$ is a subset of $G$ a set of the form $(-x + \L)\cap K$ with $x\in L$ and $K$ compact is  called a \textit{patch of $\L$}.  A set  $\L$  in $G$ is said to have \textit{finite local complexity} (FLC) if for any compact $K$ in $G$   the set
$$\{ (- x + \L) \cap K : x\in \L\}$$
is finite.  This just means that there are only finitely many patches for fixed 'size' $K$.   The relevance of sets with finite local complexity in our setting comes from the following feature of the topology.

\begin{lemma}\label{lem-top-flc} Let $G$ be a locally compact abelian group. Let $\L$ be a Delone set with finite local complexity in $G$, $X = X (\L)$ its  hull and $\Xi(\L)$ its discrete hull.
The unique uniform structure on the discrete hull compatible with the relative topology has a fundamental system given by the subsets
$$U_C := \{(\L'',\L')\in \Xi\times \Xi:  \L'' \cap C = \L^{'} \cap C\}$$
for $C\subset G$ compact.
In particular, if  $(\L_\iota)$ is  a net in $\Xi$ and $\L^{'}$  belongs to $X$, then
 the following assertions are equivalent:

\begin{itemize}
\item[(i)] $\L_\iota \to \L^{'}$.
\item[(ii)] For any compact $C$ in $G$ there exists an $\iota_C$ with $\L_\iota \cap C = \L^{'} \cap C$ for  all $\iota > \iota_C$.
\end{itemize}
\end{lemma}
\begin{proof} This is well known and appears - at least implicitly - in e.g.\  \cite{Schl,BM,BLM}.
We sketch a proof for completeness:  Let a compact $C$ in $G$ be given.
By definition of $\Xi$, we can assume without loss of generality that $C$ contains the point $e$. Then,  by finite local complexity,  there are only finitely many possibilities for $\L^{'} \cap C$ (for $\L^{'} \in \Xi$). This implies that for any  sufficiently small neighbourhood $V$ of $e\in G$, the inclusions
$$\L_1 \cap C \subset \L_2 + V \mbox{ and } \L_2 \cap C \subset \L_1 + V$$
can only hold if
$$\L_1 \cap C = \L_2 \cap C.$$
This shows that the $U_C$, $C\subset G$ compact,  form indeed a basis of the uniformity.
Now,  the last statement characterizing  convergence is immediate.
\end{proof}

An \textit{occurrences of the patch } $(-x + \L) \cap K$ in a Delone set  $\L$ is an element of
$$ \{y \in \L : (- x + \L)\cap K \subset (-y + \L)  \}.$$

A Delone set $\L$ is called  \textit{repetitive}  if for any patch  the set of occurrences is relatively dense. For FLC Delone sets this condition is equivalent to minimality of the associated system $(X(\L), G)$, but for non-FLC this is no longer the case.

For any uniformly discrete subset $\L$ of $G$ and any $\varphi \in C_c (G)$ we can define the function $$f_{\L,\varphi} : G\longrightarrow \C, \:\; f_{\L, \varphi} (t) = \sum_{x\in \L} \varphi (t - x).$$

\medskip

From these definitions and Theorem \ref{char-ap}, we immediately obtain the following theorem.

\begin{theorem} \label{thm-Delone-ec-general}  Let $\L$ be a Delone set in the locally compact $\sigma$-compact  abelian group $G$. Then, the following assertions are equivalent:

\begin{itemize}
\item[(i)] The function $ f_{\L,\varphi}$ is Bohr-almost periodic for any $\varphi \in C_c (G)$.
\item[(ii)] The measure $\delta_\L$ is strongly almost periodic.
\item[(iíi)] The hull $(X (\L), G)$ is  a compact abelian group   with  neutral element $\L$ and group addition  satisfying $(t + \L) + (s + \L) = (s +t + \L)$ for all   $t,s\in G$.
\item[(iv)]  The dynamical system $(X(\L),G)$ is equicontinuous.
\end{itemize}
\end{theorem}

We finish this section by discussing how the above
 above considerations can be  carried over to colored subsets of groups:

  Fix a finite set $\mathcal{C}$.
  We think of elements of $\mathcal{C}$ as colors.
  A \textit{colored set} (with colors from   $\mathcal{C}$)  on the locally compact abelian group $G$ is a pair $(\L, c)$ consisting of a subset $\L$ of $G$  and a map $c : \L \longrightarrow \mathcal{C}$. The set $\L$ is called the \textit{support} of the colored set. A colored set is called \textit{uniformly discrete, relatively dense, and  Delone} respectively  if and only if its support has the corresponding property.    A \textit{colored patch}  (with colors from $\mathcal{C}$)  in a colored  set $(\L,c)$ is a pair $(P,c_P)$ consisting of a patch $P$ in $\L$ and the restriction $c_P$  of $c$ to $P$. The notion of finite local complexity then carries directly over to colored  sets. As there are only finitely many colors a  colored set  $(\L,c)$ has finite local complexity if and only if  $\L$ has finite local complexity.  A colored set whose support is uniformly discrete    can be identified with a translation bounded measure on the following way: Chose an injective function $ f: \mathcal{C} \longrightarrow \C$. Then,  we can associate to $(\L,c)$ the measure
  $$\delta_{(\L,c)} := \sum_{x\in \L} f (c(x)) \delta_x.$$
  This allows one to transfer results from measures to  colored Delone sets in the same way as they were transferred from measures to Delone sets. In particular, the direct analogue of  Theorem \ref{thm-Delone-ec-general} for colored Delone sets can be shown in this way. We refrain from an explicit statement.

\section{An affirmative answer to the question of Lagarias  in the case of finite local complexity }\label{Proof}
For tilings of Euclidean space Barge and Olimb have given a beautiful
argument showing that in any FLC aperiodic repetitive tiling system
there exist
two distinct tilings which are proximal and hence
their dynamical systems are not equicontinuous.
We will present (a variant of) their reasoning for Delone sets on rather general locally compact abelian groups.

In short, the arguement of Barge and Olimb
contains two steps: First they show that  aperiodicity,
and repetitivity imply that no finite patch of a tiling can
uniquely determine this tiling. We reproduce this step  literally
in Lemma~\ref{lem-eins}. From this it is concluded that there exist
distinct tilings which agree on arbitrarily large patches
and hence are proximal. We will follow a somewhat different
argumentation which applies to Delone sets on compactly generated
locally compact abelian groups (Lemma~\ref{lem-zwei} and Corollary
\ref{cor-main}). This will allow us to give a
characterization of equicontinuous  Delone dynamical systems which
have finite local complexity (Theorem \ref{thm-Delone-FLC}) and yield
an affirmative answer to the  a question of Lagarias  for Delone
dynamical systems with finite local complexity (Corollary
\ref{cor-Lagarias-true}).

\bigskip

\begin{definition} Let $G$ be locally compact abelian group and $\L$ a Delone set in $G$. The elements of
 $P:=\{t\in G: t + \L = \L\}$ are called periods of $\L$.  The Delone set $\L$ is called completely periodic or crystalline if the set of its periods is  relatively dense in $G$. A Delone set is called aperiodic if it is not completely periodic.
\end{definition}

We say that a patch of a Delone set  \textit{forces the whole Delone set}  if any
Delone set  in its   hull   which contains the patch at the same
place is equal to the original Delone set.

\begin{lemma} \label{lem-eins}  Let $\L$ be a repetitive Delone set in a locally compact abelian group $G$, which is not completely periodic. Then, it is not forced by any of its patches.
\end{lemma}
\begin{proof} Assume the contrary. Then, there exists a patch forcing the Delone set $\L$. This means that all occurrences of this patch are periods of the Delone sets. As the Delone set is repetitive, this means that the set of its periods is relatively dense. This is a contradiction (as  $\L$ is not completely periodic).
\end{proof}

So far, we did not need any restrictions on the geometry of the underling group.  For the next lemma we will need a further restriction viz
the group $G$ needs to be compactly generated. This means that there exists a compact neighborhood $W$ of the identity, such that the smallest subgroup of $G$ containing $W$ is in fact $G$. This condition is clearly met for $G =\R^n$. Obviously, any compactly generated group is $\sigma$-compact.

\begin{lemma}\label{lem-zwei}
Let $G$ be a locally compact, compactly generated group. Let $\L$ be a
Delone set  in $G$. Then, there exists a compact set $C$ in $G$ with
the following property: For all $\L_1, \L_2$ in $X(\L)$ with $e \in
\L_1\cap \L_2$ and $\L_1 \neq \L_2$ there exists an $v\in \L_1 \cap
\L_2$ with  $$(-v + \L_1) \cap C \neq (-v + \L_2) \cap C.$$
\end{lemma}
\begin{proof}  By the main structure theorem on locally compact,
  compactly generated abelian groups, the group $G$ has the form $G =
  \R^n \times \Z^f \times H$ with a compact group $H$ and non-negative
  integer numbers $n$ and $f$.  For $r\geq 0$ let $B_r$ and $U_r$  be
  the closed and open  balls respectively   with radius $r$ around the
  origin in  $\R^n \times \Z^f$ (with respect to the usual Euclidean
  distance).   As $\L$ is a Delone set, there exists an $R>0$ such
  that
$$ \L^{'} \cap  ((t + B_R) \times H)  \neq \emptyset$$
for any $t\in \R^n\times \Z^f$ and $\L^{'} \in X(\L)$.

\medskip

To $\L_1\neq \L_2$ in  $X (\L)$ with $e \in \L_1\cap \L_2$ we can now define
$$r_M :=  \inf \{ r\geq 0 : \L_1 \cap (B_r \times H) \neq  \L_2 \cap (B_r \times H)\}.$$
From  $\L_1 \neq \L_2$ we obtain $r_M< \infty$.

\smallskip

We distinguish two cases:

\smallskip

\textit{Case 1:  $r_M = 0$:} Then, $\L_1 \cap (\{0\}\times H) \neq
\L_2 \cap (\{0\} \times H$). Here, $0$ is the origin of $\R^n \times
\Z^f$. Thus, any compact set $C$ containing $\{0\} \times H$  will do
in this case (with $v =e\in G$).

\smallskip

\textit{Case 2: $r_M >0$:} In particular, $n$ or $f$ are non-zero.
 Without loss of generality we can assume
$r_M > 2R + \epsilon$ for some $\epsilon>0$ which we determine later
(as otherwise, any compact $C$ containing $B_{2R + \epsilon}
 \times H$ will do (with $v =e\in G$)). Now, by definition of $r_M $ we
can find a $t = (s,h) \in (B_{r_M} \setminus U_{r_M})
\times H$ which belongs to only one of $\L_1$ and $\L_2$. Without loss
of generality we assume that $t$ belongs to $\L_1$ but not to $\L_2$.
On the other hand, $(U_{r_M}\times H) \cap \L_1 = (U_{r_M}\times H)\cap
\L_1\cap\L_2$.
So the idea is to find a point $v=(s',h')\in (U_{r_M}\times H) \cap \L_1$
such that the distance between $s$ and $s'$ is bounded by a constant
$c$ which depends only
on $R$, $n$ and $f$ and then choose $C = B_{c}\times H$.

To see that this is possible suppose first that $f=0$. Let $\hat s$ be the unit vector in the direction of $s$.
The $R$ ball around $x = (r_M - R - \epsilon) \hat s$ lies in
$U_{r_M}$. By definition of $R$ it contains a point $s'$ such that $v
= (s',h')\in\L_1$. Clearly $\|s-s'\|\leq 2R +\epsilon$ and so we can
take any $\epsilon>0$.

Now if $f\neq 0$ then $ (r_M - R - \epsilon) \hat s$ might not lie in
$\R^n \times \Z^f$. So we take the nearest point to it in $\R^n \times \Z^f$
such that its $R$-ball is contained in $U_{r_M}$. This requires
$\epsilon$ to be larger than $\sqrt{f}$.
\medskip

The above considerations show that we can choose
$$C = B_{2 R +  \sqrt{f}+1} \times H.$$
This finishes the proof.
\end{proof}

\textbf{Remark.} Note  that the preceding lemma neither assumes that the Delone set in question in repetitive nor that it has finite local complexity.

\bigskip

\begin{cor} \label{cor-main} Let $G$ be a compactly generated group and $\L$ an  repetitive aperiodic Delone set in $G$ with finite local complexity. Then, $(X(\L), G)$ is not equicontinuous.
\end{cor}
\begin{proof} Assume the contrary i.e.\ assume that $(X(\L),G)$ is equicontinuous.
Then the induced partial action of $G$ on $\Xi(\L)$ is equicontinuous.
By Lemma \ref{lem-top-flc} this means that
for any compact $C\subset G$ there exists a compact $C'$ such that
for all $\L_1,\L_2$ in $\Xi(\L)$, $ \L_1\cap C'=  \L_2\cap C'$ implies
$$ (-v + \L_1) \cap C = (-v + \L_2) \cap C$$
for all $v\in \L_1\cap \L_2$. Thus, if we chose  $C$ as in the previous lemma we obtain that $\L_1 = \L_2$ whenever $\L_1 \cap C' = \L_2\cap C'$. In particular, the patch $P:= \L_1\cap C'$ forces the Delone set $\L$. By Lemma \ref{lem-eins}, we then obtain that $\L$ is completely periodic. This is a contradiction.
\end{proof}

As a consequence of this corollary, we obtain the following characterization of equicontinuous Delone dynamical systems with finite local complexity.

\begin{theorem} \label{thm-Delone-FLC} Let $G$ be a compactly generated group and $\L$ a Delone set in $G$ with finite local complexity. Then, the following assertions are equivalent:

\begin{itemize}
\item[(i)]  $\L$ is completely periodic.

\item[(ii)] $(X(\L),G)$ is equicontinuous.
\end{itemize}
\end{theorem}
\begin{proof}
The implication (i)$\Longrightarrow $ (ii) is clear (and does not depend on  FLC).
It remains to show the implication (ii)$\Longrightarrow$ (i).
By the previous Corollary \ref{cor-main}, it suffices to show that $\L$ is repetitive. This, in turn,  follows from general principles as each transitive equicontinuous system must be minimal. We include the short proof for the convenience of the reader: By general principles, there will exist a minimal component $(Y,G)$ of $(X(\L),G)$. Let $\L^{'}$ be an element of $Y$. Now, as $\L^{'}$ belongs to the hull of $\L$, there exists a net $(t_i)$ in $G$ with $t_i + \L$ converging to $\L^{'}$. By the equicontinuity assumption  (ii), we infer that then $-t_i + \L^{'}$ must converge to $\L$. This shows that $\L$  belongs to   $Y$ and hence is repetitive.
\end{proof}

\medskip

\textbf{Remarks.} (a)  The theorem can also be phrased in terms of the proximality relation.

(b) The preceding theorem can be combined with Theorem \ref{thm-Delone-ec-general} to give various further characterizations of  complete periodicity of a Delone set with finite local complexity. We single out one of these in the next corollary.

\bigskip

The  previous theorem allows us to give an affirmative answer to a question  of Lagarias (under the additional assumption of finite local complexity):

\begin{cor} \label{cor-Lagarias-true} Let $G$ be a compactly generated group and $\L$ a Delone set in $G$ with finite local complexity. Then, the following assertions are equivalent:
\begin{itemize}
\item[(i)] $\L$ is completely periodic.

\item[(ii)] $\delta_\L$ is strongly almost periodic.
\end{itemize}
\end{cor}
\begin{proof} Combine the previous theorem with Theorem  \ref{thm-Delone-ec-general}.
\end{proof}

\bigskip

\textbf{Remark.}
It is not difficult to see that Lemma \ref{lem-eins} and Lemma \ref{lem-zwei} remain true for colored Delone sets. This allows one to obtain an analogue of the preceding  corollary for colored Delone sets. We leave the details to the reader.

\section{A negative answer to the question of Lagarias  for systems  without  finite local  complexity}\label{How-Failure}
In this section we show that Lagarias question does not have an affirmative answer if the assumption of finite local complexity is dropped. More precisely, we provide  examples of strongly almost periodic Delone sets on the real line  which are not crystalline. The class of examples we give can clearly be generalized to arbitrary Euclidean spaces (and even to more general locally compact abelian groups). It may be of interest in further contexts as well. For this reason we give a characterization in the one dimensional situation discussed below.

\bigskip

Our examples arise as factors of the Kronecker flow on the two-dimensional torus.  By Corollary \ref{cor-ap} such a factor  automatically yields strongly almost periodic measures. These measures will be given by Dirac combs of certain Delone sets. Thus, it suffices to show that the Delone sets in question are not crystalline. This will be clear from the construction.  Here are the details:
Let  $\T :\R^2 /\Z^2$ be the two dimensional torus
and
$$p : \R^2 \longrightarrow \T,\;\:  p((x,y)) := (x,y) + \Z^2$$
be the canonical projection.
Fix an irrational $\theta \in (0,1)$ and define
$$ h=h_\theta  : \R \longrightarrow \R^2, h (t) = (t, \theta t)$$
and
$$\iota= \iota_\theta : \R \longrightarrow \T, \iota = p \circ h_\theta.$$
Then, $\iota$ induces an action of $\R$ on $\T$ via
$$t\cdot  \xi := \iota(t) + \xi$$
 and in this way $(\T,\R)$ becomes a rotation on a compact abelian group, known as \textit{Kronecker flow} (with slope $\theta$). Let now
$$g : [0,1]\longrightarrow \R^2$$
be a continuous curve with the properties:

\begin{enumerate}
\item $g$ is everywhere transversal to the direction defined by the line $\mbox{im} (h)$. More precisely,  there exists $\epsilon>0$ such that for all $(x,y)\in \mbox{im}(g)$ we have
$(x,y)+h(t)\in \mbox{im}(g)$ only if $t=0$ or $|t|>\epsilon$.
\item   $g(1)-g(0)\in\Z^2$.
\end{enumerate}
In particular
$$\varGamma := \Z^2 + \mbox{im}  (g)$$
is the image of a closed connected curve on the torus which is transversal to the Kronecker flow. We also suppose that
\begin{itemize}
\item[(3)]\label{item3} $\varGamma$ does not intersect itself.
\end{itemize}
For example, any  continuous strictly monotone $g : [0,1]\longrightarrow [0,1]$ with $ g(0)=1$ and $g(1) = 0$ will satisfy these assumptions (1), (2) and (3).
Now, we define the set
 $$ \L_{\varGamma,\theta} = \{ t\in \R :  (t,\theta t) \in \varGamma\} = h^{-1} (\varGamma).$$

\begin{cor}\label{cor-Lagarias-wrong}  Let $g$ be non-linear and $\theta\in (0,1)$ irrational. Then, $\L_{\varGamma,\theta}$ is Delone set with  a strongly almost periodic Dirac comb which is not completely periodic.
\end{cor}
\begin{proof} The construction directly gives that  $\L_{\varGamma,\theta}$ is a Delone set:
Since $\varGamma$ is a closed curve transversal to the Kronecker flow
the set $\L_{\varGamma,\theta}$ is relatively dense and since $\varGamma$ does not intersect itself
$\L_{\varGamma,\theta}$ is uniformly discrete.

\smallskip

We now show that   complete periodicity of
$\L_{\varGamma,\theta}$   implies that $g$ is  linear:  Assume without loss of generality that $0$ belongs to $\L_{\varGamma, \theta}$. Then, the orbit $P$  of $0$ under the set of periods of $ \L_{\varGamma, \theta}$ is just  this set of periods and  hence  a group. Let $S^{'}$ be the image of $P$ under the map $\iota$. Then, $S^{'}$ is a subgroup of $\T$. Furthermore by construction $S^{'}$ is contained in the compact image of $\mbox{im} (g)$ under $p$.  Denote the closure of $S^{'}$ by $S$. Then, $S$ is   again is a subgroup and  contained in the image of $\mbox{im} (g)$ under $p$. Hence, $S$ can not be $\T$. On the other hand, as $\theta$ is irrational, the subgroup $S$ of $\T$ must contain an accumulation point. Thus, $S$ can not be discrete. Thus, $S$ must be a circle. The  image of $g$ under $p$ must then contain this circle. As this image is connected  it must then agree with this circle. Hence, $g$ is linear.

\smallskip

It remains to show the claim on strong almost periodicity:
Define the map $\widetilde{\pi}$ from $\R^2$ to the measures on   $\R$ via
$$\widetilde{\pi} (x,y) := \sum_{ t : (x,y) + h(t) \in \varGamma} \delta_t.$$
In particular, $\widetilde{\pi} (0,0) = \delta_{\L_{\varGamma,\theta}}$.
Then, $\widetilde{\pi}$ can easily be seen to have the following three properties:
\begin{itemize}
\item $\widetilde{\pi} (x,y)$ is translation bounded (as $\varGamma$ is transversal to the Kronecker flow).
\item  $\widetilde{\pi}$ is continuous (as $\varGamma$ is a continuous closed curve).
\item $\widetilde{\pi} ( (x,y) ) = \widetilde{\pi}( (x,y) + (n,m))$ for all $(x,y)\in\R^2$ and $(n,m)\in \Z^2$ (as our construction is invariant under shifts by $\Z^2$).
 \item $\widetilde{\pi} ( h(t) + (x,y)) = \alpha_{-t} \widetilde{\pi} (x,y)$ (as follows from a direct computation).
\end{itemize}
Thus, $\widetilde{\pi}$ induces a continuous map $\pi :\T \to \mathcal M^\infty(\R)$
with $\pi (t \cdot \xi) = \alpha_{-t}  \pi (\xi) $ for all $\xi \in\T$ and $t\in\R$.
By Corollary \ref{cor-ap}, the measure $\pi((0,0) + \Z^2)=\delta_{\L_{\varGamma,\theta}}$ is then strongly almost periodic.
In this way, we have therefore  constructed a Delone set $\L_{\varGamma,\theta}$ whose Dirac comb is strongly almost periodic.
\end{proof}


To complete the above discussion we determine the dynamical system $(X,\R)$ of $\L_{\varGamma,\theta}$,
or, what is the same, $\delta_{\L_{\varGamma,\theta}}$. By
Corollary~\ref{cor-ap} $\pi$ is a group homomorphism onto its image and this image is $X$. Let $K=\ker\pi$.
Then $(X,\R)$ is the rotation on $\T/K$ induced by the Kronecker flow.

Clearly $\pi((x,y)+(k_1,k_2)) = \pi((x,y))$ if and only if  $(k_1,k_2)+\varGamma = \varGamma$ i.e.\ $K$ is the stabiliser of $\varGamma$.
Given that $K$ is a closed subgroup of $\T^2$ we can distinguish essentially two possibilities depending on whether it is discrete or not.
As $\varGamma$ is connected and non intersecting  $K$ must be a cyclic sub-group of the torus in the discrete case, or a circle in the other case.

If $K$ is not discrete it actually coincides with a translate of $\varGamma$. It then must intersect the line $\mbox{im}(h)$ more then once. It
follows that the restriction of $\pi$ to the line (which corresponds to the orbit of $\nu$) is not injective.
This implies that
$\nu$ (and  $\L_{\varGamma,\theta}$) are completely periodic and so $\pi$ maps its orbit to $S^1$. Since $S^1$ is closed, $X=S^1$ in this case.
If $K$ is discrete then $X$ is a two dimensional torus again  and consequently $\nu$ is  aperiodic.

\medskip

\textbf{Remarks.} (a) Since the Kronecker flow is dense in $\T^2$ distinct subgroups $K$ yield systems which are not topologically conjugate.  The finite cyclic sub-groups of $\T^2$ thus classify the possible hulls of aperiodic Delone dynamical systems up to conjugacy which arise as factors of the Kronecker flow in the above manner.
Likewise this classification is given for periodic Delone sets by the sub-groups of $\T^2$ which are circles.

(b) If one drops the requirement that $g$ is continuous or allows for $g(1)-g(0)\neq \Z^2$,
then the above analysis completely breaks down. In fact, the Sturmian sequences associated to $\theta$ are obtained if one takes $\mbox{im}(g)$ to be the closed interval in the orthocomplement of $\mbox{im}(h)$ which is obtained by projecting (orthogonally) the unit cube onto that  orthocomplement. The resulting Delone set is
aperiodic and has FLC, and is  thus not uniformly almost periodic!

\section{A further look at Delone sets in Euclidean space}\label{Further}
Delone sets of Euclidean spaces are of particular importance for the theory both from the point of view of geometry and of physics. Thus, we have a closer look at these in this section. The main advantage is that the space of all
Delone sets of a given Euclidean space carries a well-known metric and with the help of this metric we
can formulate more directly  what it means for such a Delone set to be equicontinuous.

\bigskip

Let $\rho$ be the standard Euclidean metric on $\R^n$ i.e.
$$ \rho(x,y) = \left(\sum_{j=1}^n |x_j - y_j|^2\right)^{1/2}.$$
Denote the closed ball (with respect to this  metric)  around the origin with radius $r$ by $B_r$
and the open ball with $U_r$. Then, a subset $\L$ of $\R^n$ is a uniformly discrete if there exists an $r>0$ with
$$(x + U_r) \cap (y + U_r) = \emptyset$$
for all $x,y\in \L$ with $x\neq y$. The largest such $r$ is denoted by  $r_{\min}$ and called the \textit{packing radius} of $\L$. A subset $\L$ of $\R^n$ is  relatively dense  if there exists an $r>0$ with
$$ \R^n = \bigcup_{x\in \L} (x + B_r).$$
The smallest such $r$ is denoted by $r_{\max}$ and called the \textit{covering radius} of $\L$.

The intersection of $B_r$ with a set $\L$ is denoted by $B_r[\L]$ and called an \textit{$r$-patch}. The Hausdorff distance  on compact subsets of $\R^n$ is denoted by $d_H$. It is not hard to see that then
\begin{equation}\label{metric}
d (\L_1, \L_2) := \sup \{r >0 : d_H (B_r [\L_1] \cup \partial B_r, B_r [\L_2]\cup \partial B_r) \leq \frac{1}{r}\}
\end{equation}
gives a metric on the set of all Delone sets and that the induced topology  agrees with  the topology induced by considering the Delone sets as measures
\cite{BL,BL2,BHZ,FHK,LSto}.

We will be concerned with dynamical systems $(X,\R^n)$ with $X = X(\L)$ arising as the hull of a Delone set $\L$. Recall from Section \ref{Application} that the  canonical transversal or discrete hull of $\L$ denoted by $\Xi := \Xi (\L)$ is  defined by
$$\Xi :=\{ \L' \subset X(\L) : 0\in \L'\}.$$

The metric $d$ allows us now to use some notions introduced in Section~\ref{Maximal} for such dynamical systems $(X (\L), \R^n)$. This is discussed next:
For a Delone set $\L$, the set of  $r$ return vectors is  given by
\begin{eqnarray*} \Rr(\L,r) & = & \{a\in \R^n : d (\L, -a +  \L ) \leq \frac{1}{r}\}\\
 &=& \{a\in\R^n: d_H(B_r[\L]\cup\partial B_r, B_r[\L-a]\cup\partial B_r)\leq \frac{1}{r}\}.
 \end{eqnarray*}
(If the set $\Rr(\L,r)$  is is relatively dense for all $r>0$ the set $\L$ is called \textit{rubber repetitive}.
 For Delone sets with finite local complexity, of course,  rubber  repetitivity  coincides with repetitivity. In general, rubber  repetitive is equivalent to minimality of the dynamical system.)   A  Delone set $\L$ is  now uniformly almost periodic if for all $r>0$ the set
$$ A:= \bigcap_{\L^{'} \in X (\L)} \Rr(\L',r)$$
is relatively dense. Here comes our theorem:

\begin{theorem} \label{thm-Delone-ec} Let $\L$ be a Delone set in $\R^n$.
Then, the following assertions  are equivalent:
\begin{itemize}
\item[(i)] $\L$ is uniformly almost periodic.
\item[(ii)] For all $r>0$ the set $ \bigcap_{\L'\in\Xi(\L)} \Rr(\L',r)$ is relatively dense.
\item[(iii)] For all $r>0$ the set $ \bigcap_{x\in\L} \Rr(\L-x,r)$ is relatively dense.
\item[(iv)]  For all $\epsilon>0$ the set
$$ \bigcap_{x\in \L} \big( \L-x+B_\epsilon(0) \big)\cap
\bigcap_{x\in \L} \big( -\L+x+B_\epsilon(0) \big)
$$
is relatively dense.
\item[(v)] For all $\epsilon>0$ there exists a relatively dense set $A$ such that for all $a\in A$ there is a bijection $f_a:\L\to\L$ satisfying $|f_a(x)-(x+a)|\leq \epsilon$.
\item[(vi)] The dynamical system $(X (\L),\R^n)$   is equicontinuous.
\end{itemize}
\end{theorem}
\begin{proof} Denote the covering radius of $\L$ by   $r_{\max}$ and the packing radius by $r_{\min}$.

\smallskip

The equivalence (i)$\: \Leftrightarrow\:$(vi) is just Theorem \ref{thm-Auslander-equi}.

\smallskip

  (i)$ \:\Leftrightarrow\:$ (ii).
Clearly it suffices to consider large $r$. If $|x|\leq r_{\max}$ and $r$ is large then $\Rr(\L-x,r)\subset \Rr(\L,r-r_{\max})$. It follows that  $ \bigcap_{\L'\in X(\L)} \Rr(\L',r)$ is relatively dense  (for all $r>0$) if and only if  $ \bigcap_{\L'\in\Xi(\L)} \Rr(\L',r)$ is relatively dense (for all $r >0$).

\smallskip

(ii)$ \:\Rightarrow\:$ (iii).
This is clear, since $\{\L-x:x\in\L\}$ is a subset of $\Xi(\L)$.

\smallskip

 (iii) $ \:\Rightarrow\:$ (ii).
As the  discrete hull $\Xi(\L)$ is the closure of the set
$\{\L-x:x\in\L\}$, we can find  for any  $r>0$, $ \L'\in\Xi(\L)$ and $ c>1$  an $ x\in\L$ with
$$d_H(B_{cr}[\L-x],B_{cr}[\L'])\leq\frac1{cr}.$$

Let $a\in\Rr(\L-x,r)$ and $|a|\leq (c-1)r$. Then, for $\L'$ and $x$ as above, \begin{eqnarray*}d_H(B_r[\L'-a],B_r[\L'])& \leq &
d_H(B_r[\L'-a],B_r[\L-x-a])\\
& & +d_H(B_r[\L-x-a],B_r[\L-x]) \\& & +d_H(B_r[\L-x],B_r[\L'])\\
& \leq & \frac3r.
\end{eqnarray*}
 Thus $a\in \Rr(\L',\frac{r}3)$. It follows that
$\bigcap_{x\in\L} \Rr(\L-x,r)\cap B_{(c-1)r}(0)$ is contained in
$\bigcap_{\L'\in\Xi(\L)} \Rr(\L',\frac{r}3)$.
Letting $c\to\infty$ allows to conclude the desired statement.

\smallskip

 (iii)$ \:\Rightarrow\:$ (iv). Let $a\in \Rr(\L,r)$, that is $d(\L-a,\L)\leq \frac{1}{r}$.
From (\ref{metric}) we conclude  that $a\in \L+B_{\frac1r}(0)$ and hence $\Rr(\L,r)\subset \L+B_\frac{1}{r}(0)$. In particular there exist an  $\tilde a\in\L$ such that $|a-\tilde a|<\frac1r$. From $d(\L-a,\L)\leq \frac{1}{r}$ we then obtain
$$
d(\L - \tilde a  + a, \L- \tilde a) \leq |\tilde a - a|  + d(\L, \L - a) \leq \frac{2}{r}.$$
Thus, $a\in \Rr(\L,r)$ implies $-a\in \Rr(\L-\tilde a,\frac{r}2)$.
Hence  $$ \bigcap_{x\in\L} \Rr(\L-x,r) \subset
 \bigcap_{x\in \L} \big( \L-x+B_\frac{1}{r}(0) \big)
\cap   \bigcap_{x\in \L} \big(-\L+x+B_\frac{2}{r}(0) \big)$$ holds and the desired statement (iv) follows.

\smallskip

 (iv)$ \:\Rightarrow\:$ (v). We show that $A:=\bigcap_{x\in \L} \big( \L-x+B_\epsilon(0) \big)\cap
 \bigcap_{x\in \L} \big( -\L+x+B_\epsilon(0) \big)$ has the desired properties:
 Let $a \in A$  be given. Then, for each  $ x\in \L$ there exists an $y\in \L$ such that
 $$a+x \in y +B_\epsilon(0).$$
If $\epsilon>0$ is small enough, for instance smaller than $\frac{r_{\min}}2$,  then this $y$ is unique and this defines a function $f_a:\L\to\L$, $f_a(x) = y$ with $|f_a(x) - (a + x)|\leq \epsilon$. Clearly, for small enough $\epsilon >0$ such as e.g.  $0 < \epsilon < \frac{r_{\min}}{3} $ the function  $f_a$ is injective. Since $A=-A$ we can also construct  $f_{-a}$  which is easily seen to be the inverse of $f_a$.

\smallskip

(v) $\:\Rightarrow$ (iii).  Let $\epsilon>0$. Given any $r$-patch $P\subset \L$ we have $f_a(P) \subset \L$ and $d_H\big((P+a)\cap\partial B_r(0),f_a(P) \cup\partial B_r(0))\leq \epsilon$.   It follows that, if $P$ is an $r=\frac1\epsilon$-patch of $\L$ at $x$ then $ d_H\big((P-x)\cap\partial B_r(0),(f_a(P)-x-a) \cup\partial B_r(0)\big)\leq \frac1r$. Now since $f_a$ is bijective and $\epsilon$-close to the translation by $a$ it is, perhaps
up to an irrelevant difference near $\partial B_{r}(f_a(x))$,
 the $r$-patch of $\L$ at $f_a(x)$ we have  $ d_H\big((f_a(P)-x-a) \cup\partial B_r(0),
 B_r[\L-x-a]\cup\partial B_r(0) \big)\leq \frac1r$ and thus
 $a\in\Rr(\L-x,\frac{r}2)$.
\end{proof}

\textbf{Remarks.} (a) Note that the theorem above has some overlap with Theorem~\ref{thm-Delone-ec-general}.
Indeed, (v) of the above theorem easily gives the uniform almost periodicity of all functions of the form $\sum_{x\in \L} \varphi (\cdot  - x)$ with $\varphi \in C_c (\R^n)$ and this is statement (i) of Theorem~\ref{thm-Delone-ec-general} in the more restrictive case $G=\R^n$.

(b) The theorem can be used to provide a weak version of the so-called C\'{o}rdoba theorem \cite{Cord}. This theorem states that for a  uniformly discrete  set $\L$ of the form $\L = \cup_{j=1}^m \L_m$ and $w_1, \ldots, w_m \in \C$  the Fourier transform (taken in the sense of tempered distributions) of the generalized Dirac comb $\sum_{j=1}^m \sum_{x\in \L_j}  w_j \delta_x$ is a translation bounded pure point measure if and only if each $\L_j$ is a lattice. In our context, we can show that   for $\L$ with finite local complexity the Fourier transform of $\delta_\L$ (taken as a tempered distribution) is a pure point measure if and only if $\L$ is completely periodic. Here, the 'if' part is clear from the Poisson summation formula  and the 'only if' part follows from the theorem as  pure pointedness of the Fourier transform of a measure implies almost periodicity of the measure \cite{GdeL}.

\bigskip

\medskip

\textbf{Acknowledgments.}  Daniel Lenz would like to thank Nicolae Strungaru for most illuminating discussions on almost periodicity for measures and Meyer sets. Johannes Kellendonk would like to thank Marcy Barge for
explaining to him the importance of proximality for tiling systems.

\end{document}